\documentclass[11pt]{article}

\usepackage{amsmath,amsthm,verbatim,amssymb,amsfonts,amscd, graphicx}
\usepackage{graphicx}
\usepackage{dsfont}
\usepackage{mathrsfs}
\usepackage{mathtools}
\usepackage{enumerate}
\usepackage{mathtools}
\usepackage{lipsum}         
\usepackage{xargs}                      
\usepackage{enumitem}
\usepackage{authblk}
\usepackage{cite}
\usepackage{hyperref}

\theoremstyle{plain}

\newtheorem{proposition}{Proposition}

\theoremstyle{definition}

\begin{document}

\title{Exact and approximate solutions to the Helmholtz, Schrödinger and wave equation in $\mathbf{R}^3$ with radial data}

\author{Adrian Kirkeby\footnote{Department of Mathematics, University of Oslo. E-mail: adriankir@math.uio.no}}

\vspace{2mm}

\date{}
\maketitle
\begin{abstract}
We derive simple-to-evaluate, closed-form solutions to the inhomogeneous Helmholtz equation, $\Delta u + k^2 u = \chi_{B_{x_0,r}} $, the Schrödinger equation, $i\hbar \partial_t u + \frac{\hbar^2}{2m}\Delta u = 0$ with initial data ${u(x,0) = \chi_{B_{x_0,r}} }$, and the Cauchy problem for the linear wave equation, ${\partial_t^2 u - c^2 \Delta u = 0 }$ with initial data $\left(u(x,0),\partial_t u(x,0)\right) = \left(\chi_{B_{x_0,r}},\chi_{B_{x_0,r}} \right). $ The function $\chi_{B_{x_0,r}}$ is the characteristic function on the ball $B_{x_0,r} = \{x \in \mathbf{R}^3 : |x_0 -x| \leq r \} $. Furthermore, we use these solutions to construct explicit approximate solutions when the data are radial functions on $B_{x_0,r}$, and give various error estimates on these approximations. 

\end{abstract}

\section{Introduction}

The Helmholtz, Schrödinger and wave equation are well known, fundamental partial differential equations. The Helmholtz equation models the propagation of monochromatic waves, i.e., waves with a fixed temporal frequency, and can be applied to the study of acoustic and electromagnetic wave propagation. The Cauchy problem for the wave equation model the time-dependent propagation of waves due to initial disturbances. The Schrödinger equation governs the probabilistic evolution of particles in quantum mechanics. These equations have been thoroughly analyzed many times; see for example \cite{kirsch2016mathematical,coltonkress,eskin2011lectures} on the Helmholtz equation, \cite{evans2010partial,craig2018course,eskin2011lectures,rauch2012hyperbolic} on the wave equation and \cite{teschl2009mathematical,susskind2014quantum,eskin2011lectures,craig2018course,dennis2015princeton} on the Schrödinger equation. \\
Closed form solutions to wave equations are useful for multiple reasons, for example in resolution and uncertainty analysis in scattering problems \cite{de2016limits,griesmaier2017uncertainty}, synthetic data generation in inverse problems \cite{kirkeby2020stability}, regularity estimates \cite{eskin2011lectures,coltonkress}, perturbation methods for non-linear problems and qualitative analysis of wave fields. Due to their oscillating nature, wave equations are demanding to deal with computationally, especially for high-frequency waves and large domains, see for example \cite{runborg2007mathematical,bao2004numerical,babuska1997pollution,jin2011mathematical} and references therein. As a consequence, closed form solutions are valuable for convergence testing and analysis of numerical methods.
\newline 

In this paper we use a method that relies on the spatial symmetry of fundamental solutions to construct closed form solutions to these equations in $\mathbf{R}^3$, when the data, i.e., the initial conditions or the source term, is a characteristic function on a ball with arbitrary location and radius. The main results are found in Propositions \ref{helm}-\ref{wave} in Section \ref{2}. In Section \ref{approxx} we show how these solutions can be used to construct approximate solutions when the data is a function with radial symmetry on such balls. Since all equations are linear, the results imply the construction of solutions to equations when the data is any finite sum of such characteristic functions. Although the literature on these equations is vast, we believe the results to be novel. \\

The idea behind this paper originated while trying to generate non-trivial, high-frequency solutions for an inverse problem for the Helmholtz equation in \cite{kirkeby2020stability}.  \\
\newline 

\section{Results}
\label{2}
This section contains solutions to equations followed by their derivations. In the first subsection on the Helmholtz equation, we show in detail the method used in all computations. 
\subsection{Helmholtz equation}
 
We consider the inhomogeneous Helmholtz equation,  \newline
\begin{equation}
\left\{\begin{array}{rcl}
(\Delta + k^2)u_k &=& \chi_{B_{x_0,r}},   \quad x \in \mathbf{R}^3, \\   
 \lim\limits_{\substack{|x| \to \infty}} |x|(\partial_{|x|} -i k )u_k &=& 0,  \quad  \text{uniformly for } x/|x| \in S^{2}.
\end{array}\right.
\label{fsp}
\end{equation}
Here, $k = \frac{\omega}{c} > 0$ is the wavenumber\footnote{For the case of complex valued $k$, see below.}, where $c$ is the wave speed of the medium and $\omega$ is the temporal frequency of the wave.  $S^2$ is the unit sphere in $\mathbf{R}^3$, and the Sommerfeld radiation condition guarantees a unique, radiating solution $u_k$ (cf. \cite{eskin2011lectures}, p. 91). 
The source term is the characteristic functions $\chi_{B_{x_0,r}}(x)$, defined as
\begin{equation}
    \chi_{B_{x_0,r}}(x) = \begin{cases}
        &1, \text{ for } x \in B_{x_0,r}, \\
        &0, \text{ for } x \in \mathbf{R}^3\setminus B_{x_0,r},
    \end{cases}
\end{equation}
where $B_{x_0,r} = \{ x \in \mathbf{R}^3 : |x_0 - x| \leq r \}$ is a closed ball of radius $r$ centered at $x_0$.

We now present the first result.
\begin{proposition}
\label{helm}
Let $d = |x - x_0|$. The solution to \eqref{fsp} is given by 
\begin{equation}
    u_k(x)= \begin{cases} &\dfrac{(i -kr)\mathrm{e}^{ik(d-r)} - (i +kr)\mathrm{e}^{ik(d+r)}}{2dk^3}, \quad  \text{for } x \in \mathbf{R}^3 \setminus B_{x_0,r}, \vspace{5mm} \\ 

                    &  \dfrac{ (i +kr)(\mathrm{e}^{ik(r-d)}-\mathrm{e}^{ik(r+d)})  - 2dk }{2dk^3}, \quad  \text{for } x \in B_{x_0,r}. 
            \end{cases}
            \label{helmsol}
\end{equation}

\end{proposition}

\begin{proof}
The solution to \eqref{fsp} is given by the convolution 
\begin{equation}
 u_k(x) = \int_{B_0} G_k(x-y)\chi_{B_{x_0,r}}(y) \mathrm{d}y, \quad  x \in \mathbf{R}^3.
\label{Fk}
\end{equation}
Here $G_k$ is the outgoing fundamental solution of the Helmholtz equation in $\mathbf{R}^3$ (cf. \cite{coltonkress}),
\begin{equation}
 G_k(x) = 
		\frac{\mathrm{exp}(ik|x|)}{4\pi|x|},\quad x\in\mathbf{R}^3\setminus\{0\}.
\end{equation}
We now evaluate the integral \eqref{Fk}. 
Assume first that $x \in \mathbf{R}^3 \setminus B_{x_0,r}$ and let $d = |x - x_0|$. Let $B_{x,z}$ be a ball centered at $x$ with radius $z$, where $ d-r \leq  z \leq d+r$. Let $S(z) = \partial B_{x,z} \cap B_{x_0,r}$, i.e., the part of the surface $\partial B_{x,z}$ contained in $B_{x_0,r}$. The surface area of $S(z)$ is given by $A(z) = 2\pi zh(z) $, where $h(z)$ is the height of the spherical cap (cf. \cite{zwillinger2002crc}, p. 224). Figure \ref{fig:diagram} depicts the situation. Computing the length $z_i$ from $x$ to the intersection of the spheres, we find that $h(z) = z-z_i= z\left(1 - \frac{z^2 + d^2-r^2}{2dz} \right)$. Next, we note that for $y \in \partial B_{x,z} $ the integrand is constant; $\frac{\exp(ik|x-y|)}{4\pi |x -y|} = \frac{\exp(ikz)}{4\pi z}$. We now write $\mathrm{d}y = \mathrm{d}V(z) = A(z)\mathrm{d}z$. Hence we have reduced the integral dimension from 3 to 1, and we get
\begin{align*}
   u_k(x) &= \int_{\mathbf{R}^3}  G_k(x-y) \chi_{B_{x_0,r}}(y) \mathrm{d}y 
      = \int_{B_{x_0,r}}  \frac{\mathrm{e}^{ik|x-y|}}{4\pi |x -y|} \mathrm{d}y = \int_{d-r}^{d+r} \frac{\mathrm{e}^{ikz}}{4\pi z} A(z)\mathrm{d}z \\
          &= \frac{1}{2} \int_{d-r}^{d+r} \mathrm{e}^{ikz}z\left(1 - \frac{z^2 + d^2-r^2}{2zd} \right) \mathrm{d}z \\
          &= \frac{(i-kr)\mathrm{e}^{ik(d-r)} - (i +kr)\mathrm{e}^{ik(d+r)}}{2dk^3}.
\end{align*}

\begin{figure}
    \centering
    \includegraphics[width=1\textwidth]{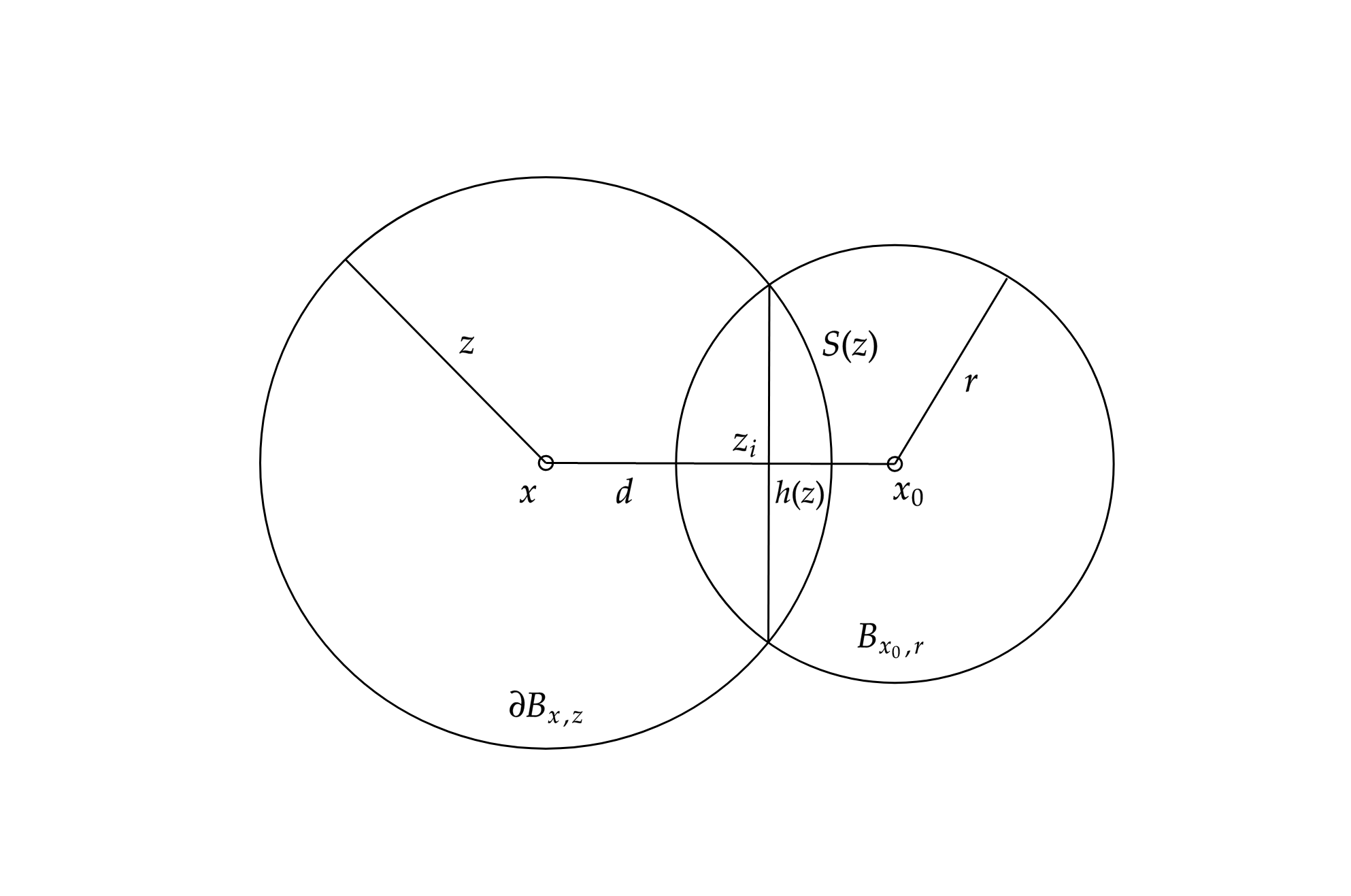}
    \caption{2-dimensional sketch of the intersecting spheres. The spherical cap $S(z)$ is the part of $\partial B_{x,z}$ contained in $B_{x_0,r}$.}
    \label{fig:diagram}
\end{figure}

Next, assume that $x \in B_{x_0,r}\setminus\{x_0\}$. We split the integral into two parts: a spherical integral when $ 0 \leq z \leq r-d$, and a spherical cap integral when $r-d \leq z \leq r+d$. 
\begin{align*}
   u_k(x) &= \int_{\mathbf{R}^3}  G_k(x-y) \chi_{B_{x_0,r}}(y) \mathrm{d}y \\
          &= \int_{B_{x_0,r}}  \frac{\mathrm{e}^{ik|x-y|}}{4\pi |x -y|} \mathrm{d}y \\
          &= \int_{0}^{r-d} \frac{\mathrm{e}^{ikz}}{4\pi z} 4 \pi z^2 \mathrm{d}z + \frac{1}{2} \int_{r-d}^{r+d} \mathrm{e}^{ikz}z\left(1 - \frac{z^2 + d^2-r^2}{2zd} \right) \mathrm{d}z\\
          &= \frac{\mathrm{e}^{ik(r-d)}(1 - ik(r-d)) - 1}{k^2} + \frac{ (i -2idk^2(d-r)+(r-2d))\mathrm{e}^{ik(r-d)} - (i +kr)\mathrm{e}^{ik(r+d)}}{2dk^3} \\
          &= \frac{ (i +kr)(\mathrm{e}^{ik(r-d)}-\mathrm{e}^{ik(r+d)})  - 2dk }{2dk^3}.
\end{align*}

Last, we have 
\begin{align*}
    u_k(x_0) &= \int_{\mathbf{R}^3}  G_k(x_0-y)\chi_{B_{x_0,r}}(y) \mathrm{d}y \\
            &= \int_0^r \frac{\mathrm{e}^{ikz}}{4\pi z} 4\pi z^2 \mathrm{d}z \\
            & = \frac{\mathrm{e}^{ikr}(1-ikr) - 1}{k^2},
\end{align*}
and it is straight forward to check that 
$$ u_k(x_0)=\lim_{d\to 0} \frac{ (i +kr)(\mathrm{e}^{ik(r-d)}-\mathrm{e}^{ik(r+d)})  - 2dk }{2dk^3} = \frac{\mathrm{e}^{ikr}(1-ikr) - 1}{k^2}.$$

\end{proof}

\textbf{\emph{Complex wavenumber:}} The above calculation also holds for complex wave numbers. If one considers instead the operator $(\Delta  + \kappa^2) $ with $\kappa^2 = k^2 + ik\sigma$, where $\sigma$ is an attenuation parameter (cf. \cite{li2021inverse}), the solution is again given by 
$$ u_{\kappa}(x) = \int_{B_0} G_{\kappa}(x-y)\chi_{B_{x_0,r}}(y) \mathrm{d}y, \quad  x \in \mathbf{R}^3, $$
where
$$
 G_{\kappa}(x) = 
		\frac{\mathrm{exp}(i\kappa|x|)}{4\pi|x|},\quad x\in\mathbf{R}^3\setminus\{0\},
$$
and $$ \text{Re}(\kappa) = \left( \frac{\sqrt{k^4 + k^2\sigma^2}+k^2}{2}\right)^{1/2}, \qquad \text{Im}(\kappa) = \left( \frac{\sqrt{k^4 + k^2\sigma^2}-k^2}{2}\right)^{1/2}.$$

As one can readily check, the solution $u_\kappa(x)$ is identical to the one in Proposition \ref{helm}, but with $k$ replaced by $\kappa$.

\subsection{Schrödinger equation}
We consider the Schrödinger equation without potential, 
\begin{equation}
    \begin{cases} &i\hbar \dfrac{\partial u}{\partial t} + \frac{\hbar^2}{2m}\Delta u = 0, \quad x \in \mathbf{R}^3, t > 0, \\
                   &u(x,0) = \chi_{B_{x_0,r}}(x), \quad x \in  \mathbf{R}^3.  
    \end{cases}
    \label{schrödinger}
\end{equation}
Here $m$ is the mass of the particle and $\hbar$ is the Planck constant. The solution $u$ is called the wave function, and $|u(x,t)|^2$ is interpreted as the probability density function of the particle; the probability that the particle is contained in some region $\Omega \subset  \mathbf{R}^3$ at the time $t$ is given by 
\begin{equation*}
    P(p \in \Omega) = \int_\Omega |u(x,t)|^2 \mathrm{d}x. 
\end{equation*}
Moreover, we have that $\|u(x,t)\|_{L^2(\mathbf{R}^3)} = \|u(x,0)\|_{L^2(\mathbf{R}^3)} $, i.e., conservation of probability. We require the total probability to sum to one at all times. Hence, for the solution to be physically meaningful, the initial condition should be multiplied by $ (4\pi r^3/3)^{-1/2}$, and $(4\pi r^3/3)^{-1}\chi_{B_{x_0,r}}(x)$ represents a uniform probability distribution on $B_{x_0,r}$ with the corresponding solution given by $(4\pi r^3/3)^{-1/2}u(x,t)$.

\begin{proposition}
\label{schro}
Let $d = |x - x_0|$ and $\mathcal{M}_t = m/2\hbar t$.  For $u(x,0) = \chi_{B_{x_0,r}}(x)$, the solution to \eqref{schrödinger} is given by 
\begin{align*}
    u(x,t)=& \frac{1}{2} \operatorname{erf}\left(\mathrm{e}^{i3\pi/4}(\mathcal{M}_t)^{1/2}(d-r)\right)-\frac{1}{2}\operatorname{erf}\left(\mathrm{e}^{i3\pi/4}(\mathcal{M}_t)^{1/2}(d+r)\right) \\
    &+ \dfrac{\mathrm{e}^{-i3\pi/4}\left(\mathrm{e}^{i\mathcal{M}_t(d-r)^2} -
    \mathrm{e}^{i\mathcal{M}_t(d+r)^2}\right)}{(\mathcal{M}_t)^{1/2}d\sqrt{\pi}}, \quad \text{for } t >0, x \in \mathbf{R}^3.
\end{align*}
\vspace{4mm}
\end{proposition}
Above, $\operatorname{erf}(z) = 2\pi^{-1/2}\int_0^z \mathrm{e}^{-t^2}\mathrm{d}t$ is the error function. 

\begin{proof}

The fundamental solution for the Schrödinger equation in $\mathbf{R}^3$ is given by (cf. \cite{eskin2011lectures,craig2018course})
\begin{equation}
    G(x,t) = \left(\frac{m \mathrm{e}^{-i\pi/2}}{2\pi \hbar  t}\right)^{3/2}\mathrm{e}^{\frac{im|x|^2}{2\hbar t}}, \quad x \in \mathbf{R}^3, t > 0,   
\end{equation}
and the solution to \eqref{schrödinger} is given by 
\begin{equation}
    u(x,t) = \int_{\mathbf{R}^3}G(x-y,t)\chi_{B_{x_0,r}}(y) \mathrm{d}y. 
\end{equation}

Proceeding as in the proof of proposition 1, we compute the above integral. We write $\mathcal{M}_t = m/2\hbar t$. For fixed $t > 0$, the observation that $G(x-y,t)$ is constant on the sphere $\partial B_{x,|x-y|}$ still holds, and for $x \in \mathbf{R}^3 \setminus B_{x_0,r}$ and  $t > 0 $ we have

\begin{align*}
    u(x,t) &=  \frac{\mathrm{e}^{-i3\pi/4}}{\pi^{3/2}}\mathcal{M}_t^{3/2}\int_{B_{x_0,r}} \mathrm{e}^{i\mathcal{M}_t|x-y|^2} \mathrm{d}y \\
    &= \frac{\mathrm{e}^{-i3\pi/4}}{\pi^{3/2}}\mathcal{M}_t^{3/2}\int_{d-r}^{d+r} \mathrm{e}^{i\mathcal{M}_tz^2} 2\pi z^2  \left(1 - \frac{z^2 + d^2-r^2}{2zd} \right) \mathrm{d}z \\
    &= \frac{\mathrm{e}^{-i3\pi/4}}{\pi^{3/2}}\mathcal{M}_t^{3/2}\Bigg(-\frac{(-1)^{1/4}\pi^{3/2} i\operatorname{erf}\left(i(-1)^{1/4}\mathcal{M}_t^{1/2}z\right)}{2\mathcal{M}_t^{3/2}}\Bigg|_{d-r}^{d+r} \\
    &\quad + \frac{i\pi\mathrm{e}^{i\mathcal{M}_tz^2}(\mathcal{M}_t(d^2 - 2dz -r^2 + z^2) + i) }{2d\mathcal{M}_t^{2}}\Bigg|_{d-r}^{d+r} 
     \Bigg) \\
    & = \frac{1}{2}\Bigg( \operatorname{erf}\left(\mathrm{e}^{i3\pi/4}(\mathcal{M}_t)^{1/2}(d-r)\right)-\operatorname{erf}\left(\mathrm{e}^{i3\pi/4}(\mathcal{M}_t)^{1/2}(d+r)\right) \\
    & \quad + \frac{2\mathrm{e}^{-i3\pi/4}\left(\mathrm{e}^{i\mathcal{M}_t(d-r)^2} -
    \mathrm{e}^{i\mathcal{M}_t(d+r)^2}\right)}{(\mathcal{M}_t)^{1/2}d\sqrt{\pi}} \Bigg).
\end{align*}

For $x \in B_{x_0,r} \setminus \{x_0\}$ and  $t > 0 $ we compute 
\begin{align*}
   u(x,t) &=  \frac{\mathrm{e}^{-i3\pi/4}}{\pi^{3/2}}\mathcal{M}_t^{3/2}\int_{B_{x_0,r}} \mathrm{e}^{i\mathcal{M}_t|x-y|^2} \mathrm{d}y  \\
    &= \frac{\mathrm{e}^{-i3\pi/4}}{\pi^{3/2}}\mathcal{M}_t^{3/2}\left( \int_0^{r-d}  \mathrm{e}^{i\mathcal{M}_t z^2} 4\pi z^2 \mathrm{d}z + \int_{r-d}^{r+d} \mathrm{e}^{i\mathcal{M}_t z^2} 2 \pi z^2  \left(1 - \frac{z^2 + d^2-r^2}{2zd} \right) \mathrm{d}z \right) \\
    &= \frac{\mathrm{e}^{-i3\pi/4}}{\pi^{3/2}}\mathcal{M}_t^{3/2}\Bigg( -\frac{(-1)^{1/4}\pi^{3/2}\operatorname{erf}\left(i(-1)^{1/4}\mathcal{M}_t^{1/2}z\right)}{\mathcal{M}_t^{3/2}}\Bigg|_0^{r-d}  - \frac{2\pi i z\mathrm{e}^{i\mathcal{M}_t z^2}}{\mathcal{M}_t}\Bigg|_0^{r-d} \\
    &\quad - \frac{(-1)^{1/4}\pi^{3/2} i\operatorname{erf}\left(i(-1)^{1/4}\mathcal{M}_t^{1/2}z\right)}{2\mathcal{M}_t^{3/2}}\Bigg|_{r-d}^{r+d} + \frac{i\pi\mathrm{e}^{i\mathcal{M}_tz^2}(\mathcal{M}_t(d^2 - 2dz -r^2 + z^2) + i) }{2d\mathcal{M}_t^{2}}\Bigg|_{r-d}^{r+d} \Bigg) \\
    & = - \operatorname{erf}\left(\mathrm{e}^{i3\pi/4}\mathcal{M}_t^{1/2}(r-d)\right) - \frac{2i\mathrm{e}^{-i\pi3/4}\mathcal{M}_t^{1/2}(r-d)\mathrm{e}^{i\mathcal{M}_t(r-d)^2}}{\sqrt{\pi}} - \frac{\operatorname{erf}\left(\mathrm{e}^{-i3\pi/4}\mathcal{M}_t^{1/2}(r+d)\right)}{2} \\
    &\quad + \frac{\operatorname{erf}\left(\mathrm{e}^{i3\pi/4}\mathcal{M}_t^{1/2}(r-d)\right)}{2} + \frac{\mathrm{e}^{-i3\pi/4}\left(\mathrm{e}^{i\mathcal{M}_t(r-d)^2} -
    \mathrm{e}^{i\mathcal{M}_t(r+d)^2}\right)}{\mathcal{M}_t^{1/2}d \sqrt{\pi}} \\
    & \quad -\frac{i\mathrm{e}^{-i\pi3/4}\mathcal{M}_t^{1/2}\mathrm{e}^{i\mathcal{M}_t(r-d)^2}(4d(d-r)) }{2d\sqrt{\pi}} \\
    &=  \frac{1}{2}\Bigg( \operatorname{erf}\left(\mathrm{e}^{i3\pi/4}(\mathcal{M}_t)^{1/2}(d-r)\right)-\operatorname{erf}\left(\mathrm{e}^{i3\pi/4}(\mathcal{M}_t)^{1/2}(d+r)\right) \\
    & \quad + \frac{2\mathrm{e}^{-i3\pi/4}\left(\mathrm{e}^{i\mathcal{M}_t(d-r)^2} -
    \mathrm{e}^{i\mathcal{M}_t(d+r)^2}\right)}{(\mathcal{M}_t)^{1/2}d\sqrt{\pi}} \Bigg), 
\end{align*}
where the last equality follows from the fact that $\text{erf}(-z) = - \text{erf}(z)$.
Last, we have 
\begin{align*}
    u(x_0,t) &= \frac{\mathrm{e}^{-i3\pi/4}}{\pi^{3/2}}\mathcal{M}_t^{3/2}\int_0^{r}  \mathrm{e}^{i\mathcal{M}_t z^2} 4\pi z^2 \mathrm{d}z \\
    &= \mathrm{e}^{-i3\pi/4}\left( -(-1)^{1/4}i\operatorname{erf}\left(i(-1)^{1/4}\mathcal{M}_t^{1/2}r\right) -2i\pi^{-1/2}r \mathcal{M}_t^{1/2} \mathrm{e}^{i\mathcal{M}_tr^2} \right) \\
    &= -\operatorname{erf}\left(\mathrm{e}^{i3\pi/4}(\mathcal{M}_t)^{1/2}r\right) +\frac{2\mathrm{e}^{i3\pi/4}r (\mathcal{M}_t)^{1/2}}{\sqrt{\pi}} \mathrm{e}^{i(\mathcal{M}_t)r^2}.
\end{align*}

We check that 
\begin{align}
u(x,t) &= \lim_{d\to 0} -\frac{1}{2}\Bigg( \operatorname{erf}\left(\mathrm{e}^{i3\pi/4}(\mathcal{M}_t)^{1/2}(r-d)\right)  +\operatorname{erf}\left(\mathrm{e}^{i3\pi/4}(\mathcal{M}_t)^{1/2}(r+d)\right)  \\
& - \frac{2\mathrm{e}^{-i3\pi/4}\left(\mathrm{e}^{i\mathcal{M}_t(r-d)^2} -
    \mathrm{e}^{i\mathcal{M}_t(r+d)^2}\right)}{(\mathcal{M}_t)^{1/2}d \sqrt{\pi}} \Bigg) \\
    &= -\operatorname{erf}\left(\mathrm{e}^{i3\pi/4}(\mathcal{M}_t)^{1/2}r\right) +\frac{2\mathrm{e}^{i3\pi/4}r (\mathcal{M}_t)^{1/2}}{\sqrt{\pi}} \mathrm{e}^{i(\mathcal{M}_t)r^2}. 
\end{align}

\end{proof}

\subsection{Wave equation}
The linear Cauchy problem for the wave equation is 
\begin{equation}
    \begin{cases} &\dfrac{\partial^2 u}{\partial t^2} - c^2 \Delta u = 0, \quad x \in \mathbf{R}^3, t > 0, \\
                   &u(x,0) = f(x), \quad \frac{\partial u(x,0)}{\partial t} = g(x), \quad x \in \mathbf{R}^3. 
 \end{cases}
 \label{waveeq}
\end{equation}
Here $c$ is the wave speed, $u(x,t)$ the wave amplitude and $f(x)$ and $g(x)$ is the initial configuration and velocity of the wave. 
\begin{proposition}
\label{wave}
Let $d =|x -x_0|$. The solution to \eqref{waveeq} with $(f(x),g(x)) = (\chi_{B_{x_0,r}}(x),0)$ is given by 
\begin{equation}
    u(x,t) = \begin{cases}
                &\dfrac{d-ct}{2d}\chi_{[d-r,d+r]}(ct), \quad \text{for } x \in \mathbf{R}^3 \setminus B_{x_0,r}, t > 0, \\
                & \chi_{[0,r-d]}(ct) + \dfrac{d -ct}{2d} \chi_{[r-d,r+d]}(ct), \quad  \text{for } x \in B_{x_0,r}\setminus \{x_0\}, t > 0, \\
                & \chi_{[0,r]}(ct) -t\delta(r-ct), \quad \text{for } x = x_0, t >0. 
                \end{cases}
\end{equation}

The solution to \eqref{waveeq} with $(f(x),g(x)) = (0,\chi_{B_{x_0,r}}(x))$ is given by 
\begin{equation}
    u(x,t) = \begin{cases}
                &\dfrac{t}{2}\left(1 - \frac{(ct)^2 + d^2-r^2}{2dct}\right)\chi_{[d-r,d+r]}(ct), \quad \text{for } x \in \mathbf{R}^3 \setminus B_{x_0,r}, t > 0, \\
                &  t\chi_{[0,r-d]}(ct)+
     \dfrac{t}{2}\left(1 - \frac{(ct)^2 + d^2-r^2}{2dc t}\right)\chi_{[r-d,r+d]}(ct), \quad  \text{for } x \in B_{x_0,r}\setminus \{x_0\}, t > 0 \\
                & t\chi_{[0,r]}(ct), \quad \text{for } x = x_0, t > 0. 
                \end{cases}
\end{equation}

\end{proposition}

\begin{proof}
The fundamental solution to \eqref{waveeq} with $f(x) = 0, g(x) = \delta(x)$ is 
$$G(x,t) =\begin{cases}
                    \dfrac{1}{4\pi c^2 t}\delta(|x| - ct), \quad &\text{for } x \in \mathbf{R}^3,t \geq 0, \vspace{3mm}\\
                    \quad 0 , \quad &\text{for } x \in \mathbf{R}^3,t < 0. 
\end{cases}$$
where $\delta(x)$ is the Dirac delta distribution (cf. \cite{eskin2011lectures}).
Hence, the solution to \eqref{waveeq} is given by 
\begin{align*}
u(x,t) &= \frac{\partial}{\partial t}\left(  \dfrac{1}{4\pi c^2 t}\int_{\mathbf{R}^3} \delta(|x-y| - ct) f(y) \mathrm{d}y \right) + \dfrac{1}{4\pi c^2 t}\int_{\mathbf{R}^3} \delta(|x-y| - ct) g(y) \mathrm{d}y \\
&= \frac{\partial}{\partial t}\left(  \dfrac{1}{4\pi c^2 t}\int_{|y| = 1} f(x+cty) \mathrm{d}S(y) \right) +\dfrac{1}{4\pi c^2 t} \int_{|y| = 1} g(x+cty) \mathrm{d}S(y), 
\end{align*} 
where $\mathrm{d}S(y)$ is the surface measure on $S^2$.  For $x \in \mathbf{R}^3\setminus B_{x_0,r}$ and $d = |x_0 -x|$ we compute 
\begin{align*}
     v(x,t) = \int_{|y| = 1} \chi_{B_{x_0,r}}(x + cty) \mathrm{d}S(y) = 2\pi (ct)^2\left(1 - \frac{(ct)^2 + d^2-r^2}{2dct}\right)\chi_{[d-r,d+r]}(ct).
\end{align*}
Here $\chi_{[d-r,d+r]}$ is the characteristic function on the interval $[d-r,d+r]$. For $x \in B_{x_0,r} \setminus \{x_0\} $ we have 
\begin{align*}
     v(x,t) = \int_{|y| = 1} \chi_{B_{x_0,r}}(x + cty) \mathrm{d}S(y) = 4\pi (ct)^2 \chi_{[0,r-d]}(ct)+
     2\pi (ct)^2\left(1 - \frac{(ct)^2 + d^2-r^2}{2dct}\right)\chi_{[r-d,r+d]}(ct).
\end{align*}
Last, 
\begin{align*}
     v(x_0,t) = \int_{|y| = 1} \chi_{B_{x_0,r}}(x_0 + cty) \mathrm{d}S(y) = 4\pi (ct)^2 \chi_{[0,r]}(ct).
\end{align*}

For $(f(x),g(x)) = (0,\chi_{B_{x_0,r}}) $ and $x \in \mathbf{R}^3 \setminus B_{x_0,r}, t > 0$, we get  
\begin{equation}
    u(x,t) = \frac{v(x,t)}{4\pi c^2 t}=\frac{t}{2}\left(1 - \frac{(ct)^2 + d^2-r^2}{2dct}\right)\chi_{[d-r,d+r]}(ct).
\end{equation}
Similar expressions are easily found for $x \in B_{x_0,r}$. \\
Next, we compute the solution for $(u(x,0),\partial_t u(x,0)) = (\chi_{B_{x_0,r}},0) $. For $x \in \mathbf{R}^3\setminus B_{x_0,r}, t > 0$ we 
have 
\begin{align*}
    u(x,t) &= \frac{\partial}{\partial t} \left(\frac{t}{2}\left(1 - \frac{(ct)^2 + d^2-r^2}{2dct}\right)\chi_{[d-r,d+r]}(ct) \right) \\
    &= \frac{t}{2}\left(1 - \frac{(ct)^2 + d^2-r^2}{2dct}\right)c\left(\delta(d-r - ct) - \delta(d+r - ct)\right)  + \frac{d-ct}{2d}\chi_{[d-r,d+r]}(ct)  \\
    & = \frac{d-ct}{2d}\chi_{[d-r,d+r]}(ct), 
\end{align*}
where the first term in the second line vanish due to $1 - \frac{(z)^2 + d^2-r^2}{2dz}$ having zeros at $d \pm r$. 
Solutions for $x \in B_{x_0,r}$ are obtained in the same way.

\end{proof}

\section{Approximate solutions for radial data}
\label{approxx}
We want to use the solutions above to approximate solutions when the data are radial functions supported on a ball. 
For a ball $B_{x_0,R}$, let $f(r) \in H^1(B_{x_0,R})$ be a radial function, i.e., a function of the radial coordinate $r=|x_0-x|$ only\footnote{Recall that $f$ is in $H^1(B_{x_0,R})$ if $\|f\|^2_{H^1(B_{x_0,R})} = \|f\|^2_{L^2(B_{x_0,R})} + \|\nabla f\|^2_{L^2(B_{x_0,R})} < \infty. $}. 
We want to approximate $f(r)$ by constant functions on spherical annulus regions. Define an annulus $S_{r_i,\Delta r}$ by $S_{r_i,\Delta r} = B_{x_0,r_i+\Delta r}\setminus B_{x_0,r_i}$. We define the approximation $f_N$ of $f$ by
\begin{equation}
    f_N(x) = \sum_{i=0}^{N-1} \Bar{f}_i \chi_{S_{r_i,\Delta r}}(x), \qquad \Delta r = \frac{R}{N}, \quad r_i = i\Delta r, \quad \Bar{f}_i = \frac{1}{\mu(S_{r_i,\Delta r})} \int_{S_{r_i,\Delta r}} f(x) \mathrm{d}x. 
    \label{approx}
\end{equation}
From the calculation 
\begin{align*}
\|f_N - f\|^2_{L^2(B_{x_0,R})} &= \sum_{i=0}^{N-1}\int_{S_{r_i,\Delta r}}|\bar{f}_i -f|^2 \mathrm{d}x \leq \sum_{i=0}^{N-1}(\Delta r)^2\int_{S_{r_i,\Delta r}}|f'|^2 \mathrm{d}x \\
& \leq \frac{R^2}{N^2}\|f\|^2_{H^1(B_{x_0,R})}, 
\end{align*}
where we have used the Poincaré inequality (cf. \cite{evans2010partial}), we have the following approximation estimate
\begin{equation}
    \|f_N - f\|_{L^2(B_{x_0,R})} \leq \frac{R}{N}\|f\|_{H^1(B_{x_0,R})}.
    \label{approxestimate}
\end{equation}

Now, let $u_k^{r_i}$ be the solution to the Helmholtz equation \eqref{fsp} with data $\chi_{S_{r_i,\Delta r}}$, i.e., a characteristic function on the annulus $S_{r_i,\Delta r}$. Since $\chi_{S_{r_i,\Delta r}} =  \chi_{x_0,r_i + \Delta r} -\chi_{x_0,r_i}$, the linearity of \eqref{fsp} implies that $u_k^{r_i}$ is given by the difference of the solutions \eqref{helmsol} with $\chi_{x_0,r_i \Delta r}$ and $\chi_{x_0,r_i}$ as data, respectively. For example, for $x \in \mathbf{R}^3\setminus B_{x_0,R}$, we have 
\begin{align*}u_k^{r_i}(x) &=  \dfrac{(i -k(r_i+\Delta r))\mathrm{e}^{ik(|x_0-x|-(r_i+\Delta r))} - (i +k(r_i+\Delta r))\mathrm{e}^{ik(|x_0-x|+(r_i+\Delta r))}}{2|x_0-x|k^3} \\
&- \dfrac{(i -kr_i)\mathrm{e}^{ik(|x_0-x|-r_i)} - (i +kr_i)\mathrm{e}^{ik(|x_0-x|+r_i)}}{2|x_0-x|k^3}.
\end{align*}
For $N >0$ and $f \in H^1(B_{x_0,R})$, let $f_N$ be the approximation. Inserting $f_N$ as data in \eqref{fsp}, we find that 
\begin{equation}
    u_k^N = \sum_{i=0}^{N-1} \Bar{f}_i u_k^{r_i} 
    \label{un}
\end{equation} 
is the corresponding approximate solution to \eqref{fsp}. 
Taking $u_k$ to be the solution to \eqref{fsp} with data $f$, we have that $${u_k^N - u_k = \int_{B_{x_0,R}} G_k(x-y)(f_N(y) -f(y)) \mathrm{d}y}.$$ Applying the Cauchy-Schwarz inequality, it follows that
\begin{align*}
    \|u_k^N - u_k\|_{L^\infty(\mathbf{R}^3)} &\leq \sup_{x\in\mathbf{R}^3} \int_{B_{x_0,R}}|G_k(x-y)(f_N(y)-f(y))| \mathrm{d}y\\
    &\leq \sup_{x\in \mathbf{R}^3}\|G_k(x -\cdot) \|_{L^2(B_{x_0,R})}\|(f_N -f)\|_{L^2(B_{x_0,R})} \leq \frac{R^{3/2}}{\sqrt{4\pi}N}\|f\|_{H^1(B_{x_0,R})},
\end{align*}
where the last inequality is a consequence of the estimate 
$$ \int_{B_{x_0,R}} |G_k(x-y)|^2 \mathrm{d}y \leq \frac{1}{(4\pi)^2}\int_{B_{0,R}} \frac{1}{|z|^2} \mathrm{d}z = \frac{R}{4\pi} \quad \text{for all } x \in \mathbf{R}^3.$$ 

We summarize the result in a proposition. 
\newline 
\begin{proposition}
For $N \in \mathbf{N}$, let $f_N$ be the piecewise constant approximation to a radial function $f \in H^1(B_{x_0,R})$ given by \eqref{approx}. Let $u_k^N$ and $u_k$ be solutions to \eqref{fsp} with data $f_N$ and $f$, respectively. Then 
$$  \|u_k^N - u_k\|_{L^\infty(\mathbf{R}^3)} \leq \frac{R^{3/2}}{\sqrt{4\pi}N}\|f\|_{H^1(B_{x_0,R})}.$$\\
\end{proposition}

More or less similar results can be obtained for the Schrödinger and wave equation as well; from the conservation of probability (cf. \cite{craig2018course} p. 154) we immediately have that 
\begin{equation}
    \|u_N(\cdot,t) - u(\cdot,t)\|_{L^2(\mathbf{R}^3)} = \|f^N - f\|_{L^2(\mathbf{R}^3)} \leq \frac{R}{N}\|f\|_{H^1(B_{x_0,R})},
    \label{conservation}
\end{equation}
where $u_N$ and $u$ are solutions to \eqref{schro} with data $f_N$ and $f$, respectively. For $t > 0$, a pointwise estimate is given by 
\begin{align*}
    \|u_N(\cdot,t) - u(\cdot,t) \|_{L^\infty(\mathbf{R}^3)} &\leq \sup_{x\in \mathbf{R}^3} \int_{B_{x_0,R}}|G(x -y)(f_N(y) -f(y))|\mathrm{d}y \\
    &\leq \sup_{x \in \mathbf{R}^3} \|G(x-\cdot)\|_{L^2(B_{x_0,R})}\|f_N - f\|_{L^2(B_{x_0,R})} \\
    & \leq \left(\frac{m^3}{6\pi^2 (\hbar t)^3}\right)\frac{ R^4}{N}\|f\|_{H^1(B_{x_0,R})}.
\end{align*}
Above and below, the approximations of $f_N$ and $u_N$ are constructed as in equations \eqref{approx} and \eqref{un}, but with solutions from Proposition \ref{helm} replaced by solutions from Propositions \ref{schro} and \ref{wave}. 
However, one should note that $\|f\|_{L^2} = 1 $ does not necessarily imply $\|f_N\|_{L^2} = 1 $, and hence $|u_N(x)|^2$ may not sum to one. Still, estimate \eqref{conservation} shows that by increasing $N$ we can make $u_N$ arbitrarily close to $u$ in the $L^2$-norm.   
\\

For the wave equation we can use $L^p$-estimates for Fourier integral operators (cf. \cite{sogge1993p}, Eq. 6) to conclude that  
\begin{align*} 
\|u_N(\cdot,t) - u(\cdot,t)\|_{L^2(\mathbf{R}^3)} &\leq C_T\left(\|f_N -f\|_{L^2(\mathbf{R}^3)} + \|g_N -g\|_{H^{-1}(\mathbf{R}^3)} \right) \\
&\leq C_T\left(\|f_N -f\|_{L^2(\mathbf{R}^3)} + \|g_N -g\|_{L^2(\mathbf{R}^3)}\right) \\
&\leq C_T\frac{R}{N}\left(\|f\|_{H^1(B_{x_0,R})} + \|g\|_{H^1(B_{x_0,R})}\right), \quad 0 < t < T <\infty.   
\end{align*}
Above, $u$ and $u_N$ are solutions to \eqref{waveeq} with radial initial data $(f,g)$ in $H^1(B_{x_0,R})$ and $(f_N,g_N)$, respectively, and the constant $C_T$ depends on the maximum time $T$. We summarize the results.  

\begin{proposition}
For $N \in \mathbf{N}$, let $f_N$ and $g_N$ be the piecewise constant approximations to radial functions $f,g \in H^1(B_{x_0,R})$. Let $v$ and $v_N$ be solutions to the Schrödinger equation \eqref{schrödinger} with initial data $f$ and $f_N$, respectively. Then $v_N$ satisfy the estimate

$$  \|v_N(\cdot,t) - v(\cdot,t) \|_{L^\infty(\mathbf{R}^3)} \leq  \left(\frac{m^3}{6\pi^2 (\hbar t)^3}\right)\frac{ R^4}{N}\|f\|_{H^1(B_{x_0,R})}.$$

Let $w$ and $w_N$ be solutions to the wave equation \eqref{waveeq} with initial data ${(w(x,0),\partial_t w(x,0)) = (f,g)}$ and ${(w_N(x,0),\partial_t w_N(x,0)) = (f_N,g_N)}$, respectively. Then $w_N$ satisfy the estimate 
$$\|w_N(\cdot,t) - w(\cdot,t)\|_{L^2(\mathbf{R}^3)} \leq C_T\frac{R}{N}\left(\|f\|_{H^1(B_{x_0,R})} + \|g\|_{H^1(B_{x_0,R})}\right), \quad 0 < t < T <\infty. $$

\end{proposition}

\section{Discussion}
Due to the linearity of the above equations, all results can be extended to obtain solutions when the data is any finite linear combination of characteristic functions on balls. As seen in Section \ref{approxx}, this includes characteristic functions on spherical shells, but any function that can be described by a sum $g = \sum w_i\chi_{B_{x_i,r_i}}$ will have a similar solution. Since many shapes in nature are spherical, this should have interesting applications. Moreover, it can possibly be used for approximation of more complicated functions than radial ones. 
Last, the method used to find solutions in this paper can, in principle, be generalized to any PDE with spherically symmetric fundamental solutions. However, the explicit and simple form of the surface measure on spherical caps that makes the calculation work out is, as far as we know, only available in $\mathbf{R}^3$. 
\vspace{5cm}
\subsection*{}
\emph{The author was partially funded by the Research Council of Norway project number 301538.}

\clearpage
\section*{Bibliography}
\bibliographystyle{plain}
\bibliography{main}
\end{document}